\documentclass[a4paper,12pt]{article}
\usepackage{amsthm}
\usepackage{amsfonts}
\usepackage{amsmath}
\usepackage{latexsym}
\usepackage{amssymb}
\usepackage{url}
\usepackage{cite}
\usepackage{enumerate}
\usepackage[active]{srcltx}
\usepackage{mathrsfs}
\usepackage{tikz}
\newcommand{\bt}{\mathbin{\tikz [x=1.6ex,y=1.6ex,line width=.1ex] \draw (0,0) -- (1,1) (0,1) -- (1,0);}}%

\newtheorem{theorem}{Theorem}
\newtheorem{lemma}{Lemma}

\newtheorem{proposition}{Proposition}
\newtheorem{remark}{Remark}

\DeclareMathOperator{\cone}{cone}

\DeclareMathOperator{\md}{mid}

\newcommand{\lng}{\langle}
\newcommand{\rng}{\rangle}
\newcommand{\lf}{\left}
\newcommand{\rg}{\right}

\newcommand{\R}{\mathbb R}
\newcommand{\N}{\mathbb N}

\newcommand{\Hy}{\mathcal H}
\newcommand{\f}{\frac}
\newcommand{\tp}{^\top}

\begin{document}

\title{Extended Lorentz cones and variational inequalities on cylinders 
\thanks{ 
inequalities, Picard iteration, fixed point}
}
\author{S. Z. N\'emeth\\School of Mathematics, The University of Birmingham\\The Watson Building, Edgbaston\\Birmingham B15 2TT, United Kingdom\\email: s.nemeth@bham.ac.uk\and G. Zhang\\School of Mathematics, The University of Birmingham\\The Watson Building, Edgbaston\\Birmingham B15 2TT, United Kingdom\\email: gxz245@bham.ac.uk}
\date{}
\maketitle

\begin{abstract}
	Solutions of a variational inequality are found by giving conditions for the monotone convergence with respect to a cone of the Picard 
	iteration corresponding to its natural map. One of these conditions is the isotonicity of the projection onto the closed convex set in
	the definition of the variational inequality. If the closed convex set is a cylinder and the cone is an extented Lorentz cone, then this
	condition can be dropped because it is automatically satisfied. The obtained result is further particularized for unbounded box
	constrained variational inequalities. For this case a numerical example is presented.
\end{abstract}

\section{Introduction}

In this paper we will study the solvability of variational inequalities on closed convex sets by using the isotonicity of the metric projection 
mapping onto these sets with respect to the partial order defined by the cone. Apparently this approach has not 
been considered before. 

Variational inequalities are models of various important
problems in physics, engineering, economics and other sciences. The classical Nash equilibrium concept can also be reformulated by using 
variational inequalities. 

It is known that a vector is a solution of a variational inequality if and only if it is a zero of the corresponding 
natural mapping \cite{FacchineiPang2003}. Hence, a variational inequality on a closed convex set is equivalent to finding the fixed points of 
the difference between the identity mapping and its natural mapping. The latter mapping is just the composition between the projection onto
the latter closed convex set and the difference between the identity mapping and the mapping defining the variational inequality. I we could
guarantee the isotonicity of the mappings in the latter composition with respect to the partial order defined by the cone, then by an iterative 
process we could construct 
an increasing sequence with respect to the partial order defined by the cone by using Picard's iteration. If we could also guarantee that this
sequence is bounded from above with respect to the partial order defined by the cone, then this sequence would be convergent to the solution
of the variational inequality (in fact it would be convergent to a solution of the corresponding equivalent fixed point problem). It turns out
that there is a class of variational inequalities and a class of cones, that extend the Lorentz cone, for which this idea works very well. The
only restriction is that the variational inequality has to be defined on a cylinder. Such problems appears in the practice. For example the
unbounded box constrained variational inequalities are of this form. Based on the above idea a theorem for finding solutions of variational 
inequalities on a cylinder will be presented and an example will be given. Similar ideas to the above ones were presented in
\cite{IsacNemeth1990b,IsacNemeth1990c,Nemeth2009,AbbasNemeth2011,AbbasNemeth2012,AbbasNemeth2012b} for complementarity and implicit
complementarity problems, but with a strong restriction on the cone defining the problem. The idea of monotone convergence 
for complementarity problems defined by general cones is considered first in the paper \cite{NemethZhang}. The present paper extends the 
results of \cite{NemethZhang} for variational inequalities.

Several other papers dealt with conditions of 
convergence for iterations similar to the above one, as for example 
\cite{Auslender1976,Bertsekas1989,Iusem1997,Khobotov1987,Korpelevich1976,Marcotte1991,Nagurney1993,Sibony1970,Solodov1999,Solodov1996,Sun1996}. 
However, neither of these works used the ordering defined by a cone for showing the convergence of the corresponding iterative scheme. Instead, they used as a tool the Banach fixed point theorem and assumed Kachurovskii-Minty-Browder type monotonicity 
(see\cite{Kachurovskii1960,Minty1962,Minty1963,Browder1964}) and global Lipschitz properties of $F$. 

The structure of the
paper is as follows: In the Preliminaries we introduce the terminology and notations use throughout this paper. In Section 3 we recall the
definition and basic properties of the extended Lorentz cone, and determine all sets onto which the projection is isotone with respect to this
cone. In Section 4 we will find solutions of a variational inequality by analyzing the monotone convergence with respect to a cone of the Picard 
iteration corresponding to its natural map. In Section 5 we will particularize these results to variational inequalities defined on cylinders, by
using the extended Lorentz cone for the corresponding monotone convergence above. In this case we can drop the condition of Proposition
\ref{psumm} that the projection onto the closed convex set in the definition of the variational inequality is isotone with respect to the 
extended Lorentz cone, because this condition is automatically satisfied obtaining the more explicit result of Theorem \ref{tsumm}. The latter 
result extend our results in \cite{NemethZhang} for mixed complementarity problems. In Section 6 we further particularize the results for 
unbounded box constrained variational inequalities. In Section 7 we will give a numerical example for this case. 

\section{Preliminaries}

Denote by $\mathbb{N}$ the set of nonnegative integers. Let $k,m,p,q\in\mathbb{N}\setminus\{0\}$ and $\mathbb{R}^m$ be the $n$-dimensional real 
Euclidean vector space. 

Define the direct product space $\mathbb{R}^p \times \mathbb{R}^q$ as the pair of vectors $(x,u)$, where $x \in \mathbb{R}^p$ and 
$u \in \mathbb{R}^p$. 

Identify the vectors of $\R^m$ by column vectors and consider the canonical inner product in $\R^m$ defined by 
$\R^m\times\R^m\ni (x,y)\mapsto x^\top y\in\R$ with induced norm $\R^m\ni x\mapsto \|x\|=\sqrt{x\tp x}\in\R$.

 The vector space $\mathbb{R}^p \times \mathbb{R}^q$ can be identified with $\mathbb{R}^{p+q}$ via the mapping 
$\mathbb{R}^p \times \mathbb{R}^q\ni (x,u)\mapsto (x\tp, u\tp)\tp \in \mathbb{R}^{p+q}$. This identification leads to the inner product in 
$\mathbb{R}^p \times \mathbb{R}^q$ given by
\begin{equation*}
((x,u),(y,v))\mapsto x\tp y+u\tp v
\end{equation*}

The closed set $K\subset\R^m$ will be called a \emph{cone} if $K\cap(-K)=\{0\}$ and $\lambda x+\mu y\in K$, whenever $\lambda,\mu\ge0$ and 
$x,y\in K$. Let $K\subset \mathbb{R}^{p+q}$ be a  cone. Denote $\leq_K$ the relation defined by $x\leq_K y\iff y-x\in K$ 
and call it \emph{the partial order defined by $K$}. The relation $\leq_K$ is reflexive, transitive, antisymmetric and compatible with 
the linear structure of $\mathbb{R}^m$ in the sense that $x\leq_K y$ implies that $tx+z\leq_K ty+z$, for any $z\in\mathbb{R}^m$ and any 
$t\ge0$. Moreover, $\leq_K$ is continuous at $0$ in the sense that if $x^n\to x$ when $n\to\infty$ and $0\le_K x^n$ for any 
$n\in\mathbb{N}$, then $0\leq_K x$. Conversely any reflexive, transitive and antisymmetric relation $\preceq$ in $\R^m$ which is compatible with 
the linear structure of $\mathbb{R}^m$ and it is continuous at $0$ is defined by a  cone. More specifically, $\preceq=\leq_K$, where 
$K=\{x\in\mathbb{R}^m:0\preceq x\}$ is a  cone.

For any closed convex set $C$ denote by $P_C:\R^m\to\R^m$ the \emph{metric projection mapping onto $C$}, that is, the mapping defined by  
$P_C(x)\in C$ and
\begin{equation*}
	\|x-P_C(x)\|=\min\{\|x-y\|: y \in C\},
\end{equation*}
for any $x\in\R^m$. Since $C$ is closed and convex, the projection mapping is well defined and by its definition it follows that 
\begin{equation}\label{protr}
	P_{y+C}(x)=y+P_C(x-y)
\end{equation}
for any $x,y\in\R^m$. It can be also shown that $P_C$ is nonexpansive (see \cite{Zarantonello1971}), that is, 
\begin{equation}\label{pronexp}
	\|P_C(x)-P_C(y)\|\le\|x-y\|, 
\end{equation}
for any $x,y\in\R^m$.

Let $K\subset\R^m$ be a  cone. The  mapping $F:\R^m\to\R^m$ is called \emph{$K$-isotone} if $x\le_K y$ implies 
$F(x)\le_K F(y)$. 

The nonempty closed convex set $C\subseteq\R^m$ is called $K$-isotone projection set if $P_C$  is $K$-isotone.  

The set $\Omega\subset\R^m$ is called $K$-bounded from below (\emph{$K$-bounded from above}) if there exists a vector $y\in\R^m$ such that 
$y\le_K x$ ($x\le_K y$), for all $x\in\Omega$. In this case $y$ is called a \emph{lower $K$-bound} (\emph{upper $K$-bound}) of $\Omega$.  
If $y\in\Omega$, then $y$ is called the \emph{$K$-least element} (\emph{$K$-greatest element}) of $\Omega$.

Let $\mathcal I\subset\N$ be an unbounded set of nonnegative integers. The sequence $\{x^n\}_{n\in\mathcal I}$ is called $K$-increasing 
($K$-decreasing) if 
$x^{n_1}\le_K x^{n_2}$ ($x^{n_2}\le_K x^{n_1}$), whenever $n_1\le n_2$.  

The sequence $\{x^n\}_{n\in\mathcal I}$ is called \emph{$K$-bounded from below} (\emph{$K$-bounded from above}) if the set 
$\{x^n:n\in\mathcal I\}$ is $K$-bounded from below ($K$-bounded from above).

A closed convex cone $K$ is called \emph{regular} if any $K$-increasing sequence which is $K$-bounded from above is convergent. It is easy to 
show that this is equivalent to the convergence of any $K$-decreasing sequence which is $K$-bounded from below. It is known 
(see \cite{McArthur1970}) that any  cone in $\R^m$ is regular.

The \emph{dual} of a cone $K\subset\R^m$ is the cone $K^*$ defined by \[K^*=\{x\in\R^m:x\tp y=0,\textrm{ }\forall y\in K\}.\]

A cone $K\subset\R^m$ is called \emph{polyhedral} if it is \emph{generated} by a finite number of vectors $v^1,\dots,v^k$, that is,
\[K=\cone\{v^1,\dots,v^k\}:=\{\lambda_1v^1+\dots+\lambda_kv^k:\lambda_1,\dots,\lambda_k\ge0\}.\] The vectors $v^1,\dots,v^k$ are called the
generators of $K$.

The \emph{affine hyperplane} with normal $u\in\R^m\setminus\{0\}$ and through $a\in\R^m$ is the set defined by
\begin{equation}\label{hyperplane}
	\Hy(u,a)=\{x\in \R^m:\;\lng x-a,u\rng=0\}.
\end{equation}
An affine hyperplane $\Hy(u,a)$ determines two \emph{closed halfspaces} $\Hy_-(a,u)$ and $\Hy_+(u,a)$  of $\R^m$, defined by
\[\Hy_-(u,a)=\{x\in \R^m:\lng x-a,u\rng\le0\},\]
and
\[\Hy_+(u,a)=\{x\in \R^m:\lng x-a,u\rng\ge0\}.\]

\section{Extended Lorentz cones}

Let $p,q$ be positive integers. For $a,b\in\R^p$ denote $a\ge b$ if and only if $b\le_{\R^p_+} a$, that is, components of $a$ are at least as 
large as the corresponding components of $b$. Denote by $e$ the vector in $\R^p$ with 
all components equal to one and by $e^i$ the canonical unit vectors of $\R^p$.
In \cite{NemethZhang} we defined the following notion of an \emph{extended Lorentz cone}:
\begin{equation}\label{extlor} 
	L=\{(x,u)\in\R^p\times\R^q:x\ge\|u\|e\}
\end{equation} 
and showed that the dual of $L$ is
\begin{equation*}
	L^*=\{(x,u)\in\R^p\times\R^q:x\tp e\ge\|u\|,x\ge0\}.
\end{equation*}

In the same paper we also showed the followings: 

\begin{itemize}
	\item The extended Lorentz cone $L$ defined by \eqref{extlor} is a (regular) cone. 
	\item The cone $L$ (or $L^*$ \cite{Rockafellar1970}) is a polyhedral cone if and only if $q=1$. 
	\item If $q=1$, then the minimal number of generators of $L$ is 
		\[(p+2)(1-\delta_{p1})+2\delta_{p1},\] where $\delta$ denotes the Kronecker symbol. 
	\item If $q=1$, $p=1$, then a minimal set of generators of $L$ is \[\{(1,1),(1,-1)\},\] and if $q=1$, $p>1$, then a minimal set of 
		generators of $L$ is \[\{(e,1),(e,-1),(e^i,0):i=1,\dots,p\}.\] 
	\item If $q=1$, then $L^*$ is a $p+1$ dimensional polyhedral cone with the minimal number of generators $2p$ and a minimal set of 
		generators of $L^*$ is \[\{(e^i,1),(e^i,-1):i=1,\dots,p\}.\] 
	\item If $q=1$ and $p>1$, then note that the number of generators of $L$ and $L^*$ coincide if and only if they are
		$2$ or $3$-dimensional cones. 
	\item The cone $L$ is a subdual cone and $L$ is self-dual if and only if $p=1$, that is, $L$ is the $q+1$-dimensional 
		Lorentz cone. 
	\item $L$ is a self-dual polyhedral cone if and only if $p=q=1$.
\end{itemize}

In Theorem 2 of \cite{NemethZhang} we determined the $L$-isotone 
projection sets. For convenience we repeat this theorem here:

\begin{theorem}\label{tliso}
	$\,$

	\begin{enumerate}
		\item Let $K=\R^p\times C$, where $C$ is an arbitrary nonempty closed convex set in $\R^q$ and $L$ be the extended Lorentz cone 
			defined by \eqref{extlor}. Then, $K$ is an $L$-isotone projection set. 
		\item Let $p=1$, $q>1$ and $K\subset\R^p\times\R^q$ be a nonempty closed convex set. Then, $K$ is an $L$-isotone projection set 
			if and only if $K=\R^p\times C$, for some $C\subset\R^q$ nonempty closed convex set.
		\item Let $p,q>1$, and 
			\begin{equation*}
				K=\cap_{\ell\in \N} \Hy_-(\gamma^\ell,\beta^\ell)\subset\R^p\times\R^q,
			\end{equation*}
			where $\gamma^\ell=(a^\ell,u^\ell)$ is a unit vector. Then, $K$ is an $L$-isotone projection set if and only if for each 
			$\ell$ one of the following conditions hold:
			\begin{enumerate}
				\item The vector $a^\ell=0$.
				\item The vector $u^\ell=0$, and there exists $i\ne j$ such that $a^\ell_i=\sqrt2/2$, $a^\ell_j=-\sqrt2/2$ and 
					$a^\ell_k=0$, for any $k\notin\{i,j\}$.
			\end{enumerate}
	\end{enumerate}
\end{theorem}

\section{Variational inequalities}

Let $K\subset\R^m$ be a closed convex set and $F:\R^m\to\R^m$ be a mapping. Then, the 
variational inequality $VI(K,F)$ defined by $F$ and $K$ is the problem of finding an $x^*\in K$ such that for any $y \in K$, 
\begin{equation*}
	(y-x)\tp F(x) \geq 0.
\end{equation*}
It is known that $x^*$ is a solution of $VI(F,K)$ if and only 
if it is a fixed point of the mapping $I-F_K^{\rm nat}=P_K\circ (I-F)$, where $I$ is the identity mapping of $\R^m$ and $F_K^{\rm nat}$ is the natural mapping associated to
$VI(F,K)$ defined by $F_K^{\rm nat}=I-P_K\circ (I-F)$ \cite{FacchineiPang2003}. Consider the Picard iteration 
\begin{equation}\label{picard} 
	x^{n+1}=P_K(x^n-F(x^n),
\end{equation}
If $F$ is continuous and $\{x^n\}_{n\in\N}$ is convergent to $x^*$, then by a simple limiting process in \eqref{picard}, it follows that 
$x^*$ is a fixed point of the mapping $P_K\circ (I-F)$ and hence a solution of $VI(F,K)$. Therefore, it is natural to seek convergence conditions for $x^n$. Let us first state the following simple lemma:

\begin{lemma}\label{lmoniso}
	Let $K\subset\R^m$ be a closed convex set, $F:\R^m\to\R^m$ be a continuous mapping and $L$ be a  cone. Consider the
	sequence $\{x^n\}_{n\in\N}$ defined by \eqref{picard}. Suppose that the mappings $P_K$ and $I-F$ are $L$-isotone, $x^0\le_L x^1$, and there
	exists a $y\in\R^m$ such that $x^n\le_L y$, for all $n\in\N$ sufficiently large. Then, $\{x^n\}_{n\in\N}$ is convergent and its limit $x^*$ is a solution of 
	$VI(F,K)$.
\end{lemma}
\begin{proof}
Since the mappings $P_K$ and $I-F$ are $L$-isotone, the mapping $x\mapsto P_K\circ (I-F)$ is also $L$-isotone. Then, by using \eqref{picard} and a simple inductive 
	argument, it follows that $\{x^n\}_{n\in\mathbb{N}}$ is $L$-increasing. Since any  cone in $\mathbb{R}^m$ is regular, $\{x^n\}_{n\in\mathbb{N}}$ is convergent and 
	hence its limit $x^*$ is fixed point of $P_K\circ (I-F)$ and therefore a solution of $VI(F,K)$.
\end{proof}
\begin{remark}\label{rc}

		The condition $x^0\le_L x^1$ in Lemma \ref{lmoniso} is satisfied when $x^0\in K\cap F^{-1}(-L)$. Indeed, if $x^0\in K\cap F^{-1}(-L)$, then 
			$-F(x^0)\in L$ and $x^0\in K$. Thus $x^0\le_L x^0-F(x^0)$, and hence by the isotonicity of $P_K$ we obtain $x^0=P_K(x^0)\le_L P_K(x^0-F(x^0))=x^1$.
\end{remark}
\begin{proposition}\label{psumm}
	Let $K\subset\mathbb{R}^m$ be a closed convex set, $F:\mathbb{R}^m\to\mathbb{R}^m$ be a continuous mapping and $L$ be a  cone. Consider the
	sequence $\{x^n\}_{n\in\mathbb{N}}$ defined by \eqref{picard}. Suppose that the mappings $P_K$ and $I-F$ are $L$-isotone and $x^0\le_L x^1$. Denote by $I$
	the identity mapping. Let 
\begin{equation*}
\Omega=\{x\in K \cap  (x^0 + L) :F(x)\in L\}
\end{equation*}
\begin{equation*}
\Gamma=\{x\in K \cap (x^0 + L) :P_K(x-F(x))\le_L x\}.
\end{equation*} 
Consider the following assertions:
	\begin{enumerate}[(i)]
		\item\label{i} $\Omega\ne\varnothing$,
		\item\label{ii} $\Gamma\ne\varnothing$,
		\item\label{iii} The sequence $\{x^n\}_{n\in\mathbb{N}}$ is convergent and its limit $x^*$ is a solution of $VI(F,K)$. Moreover, 
			$x^*$ is the $L$-least element of
			$\Gamma$. 
	\end{enumerate}
	Then, $\Omega\subset\Gamma$ and (\ref{i})$\implies$(\ref{ii})$\implies$(\ref{iii}).
\end{proposition}
\begin{proof}
	 Let us first prove that $\Omega\subset\Gamma$. Indeed, let $y\in\Omega$. Since $P_K$ is $L$-isotone, $y-F(y)\le_L y$ implies $P_K(y-F(y))\le_L P_K(y)=y$, which shows that $y\in\Gamma$. Hence, 
	 $\Omega\subset\Gamma$. Thus, (\ref{i})$\implies$(\ref{ii}) is trivial now. 
	\medskip

	\noindent 
	(\ref{ii})$\implies$(\ref{iii}):
	\medskip

	\noindent 
	Suppose that $\Gamma\ne\varnothing$. Since the mappings $P_K$ and $I-F$ are $L$-isotone, the mapping 
	$P_K\circ (I-F)$ is also $L$-isotone. Similarly to the proof of Lemma \ref{lmoniso}, it can be shown that $\{x^n\}_{n\in\mathbb{N}}$ is $L$-increasing. Let $y\in\Gamma$ 
	be arbitrary but fixed. We have $y-x^0\in L $, that is $x^0 \leq_L y$. Now, suppose that $x^n\le_L y$. Since the mapping $P_K\circ (I-F)$ is $L$-isotone, 
	$x^n\le_L y$ implies that $x^{n+1}=P_K(x^n-F(x^n))\le_L P_K(y-F(y))\le_L y$. Thus, we have by induction that $x^n\le_L y$ for all
	$n\in\mathbb{N}$. Then, Lemma implies that $\{x^n\}_{n\in\mathbb{N}}$ is convergent and its limit $x^*\in K\cap (x^0+L)$ is a solution of
	$VI(F,K)$. Since $x^*$ is a solution of $VI(F,K)$, we have that $P_K(x^*-F(x^*))=x^*$ and hence $x^*\in\Gamma$. 
	Moreover, the relation $x^n\le_L y$ in limit gives $x^*\le y$. Therefore, $x^*$ is the smallest element of $\Gamma$ with respect to the 
	partial order defined by $L$.
\end{proof}


\section{Variational Inequality on cylinders}

Let $p,q$ be positive integers and $m=p+q$. By a \emph{cylinder} we mean a set $K=\mathbb{R}^p\times C\subset\R^p\times\R^q\equiv\R^m$.
In this section we will particularize the results of the previous section for variational inequalities on cylinders.

\begin{lemma}\label{lmicp}
	Let $K=\mathbb{R}^p\times C$, where $C$ is an arbitrary nonempty closed convex set in $\mathbb{R}^q$. Let 
	$G:\mathbb{R}^p\times\mathbb{R}^q\to\mathbb{R}^p$, 
	$H:\mathbb{R}^p\times\mathbb{R}^q\to\mathbb{R}^q$ and $F=(G,H):\mathbb{R}^p\times\mathbb{R}^q\to\mathbb{R}^p\times\mathbb{R}^q$. Then, 
	the variational inequality $VI(F,K)$ is equivalent to the problem of finding a vector $(x,u)\in\mathbb{R}^p\times C$ such that
\begin{equation}\label{vicartesian}
\lf\{
	\begin{array}{rcl}
		G(x,u)=0,\\
		(v-u)\tp H(x,u)& \geq& 0
	\end{array}
	\rg.
\end{equation}
for any $v\in C$.
\end{lemma}
\begin{proof}
	The variational inequality $VI(F,K)$ is equivalent to finding an $(x,u)\in\mathbb{R}^p\times C$ such that
	\begin{equation}\label{vi}
		(y-x)\tp G(x,u)+ (v-u)\tp H(x,u) \geq 0
	\end{equation}
	for any $(y,v) \in \mathbb{R}^p \times C$. Let $(x,u)\in\mathbb{R}^p\times C$ be a solution of \eqref{vi}. If we choose $v=u \in C$ in 
	\eqref{vi}, then we 
	get $(y-x)\tp G(x,u) \geq 0$ for any $y\in\mathbb{R}^p$. Hence, $G(x,u)=0$ and $ (v-u)\tp H(x,u) \geq 0$. Conversely, if $(x,u)\in\mathbb{R}^p\times C$ is
	a solution of \eqref{vicartesian}, then it is easy to see that it is a solution of \eqref{vi}.
\end{proof}

By using the notations of Lemma \eqref{lmicp} the Picard iteration \eqref{picard} can be rewritten as:
\begin{equation}\label{mpicard}
	\lf\{
	\begin{array}{rcl}
		x^{n+1}&=&x^n-G(x^n,u^n),\\
		u^{n+1}&=&P_C(u^n-H(x^n,u^n)).
	\end{array}
	\rg.
\end{equation}
Consider the partial order defined by the extended Lorentz cone \eqref{extlor}. Then, we obtain the following proposition.
\begin{theorem}\label{tsumm}
	Let $K=\mathbb{R}^p\times C$, where $C$ is a closed convex set. Let $G:\mathbb{R}^p\times\mathbb{R}^q\to\mathbb{R}^p$, 
	$H:\mathbb{R}^p\times\mathbb{R}^q\to\mathbb{R}^q$ be continuous mappings, 
	$F=(G,H):\mathbb{R}^p\times\mathbb{R}^q\to\mathbb{R}^p\times\mathbb{R}^q$. Let $(x^0,u^0) \in \mathbb{R}^p\times C$ and 
	consider the sequence ${(x^n, u^n)}_{n \in \mathbb{N}}$ defined
	by \eqref{mpicard}. Let $x, y \in \mathbb{R}^p$ and $u, v \in \mathbb{R}^q$. Suppose that $x^1-x^0\ge\|u^1-u^0\|e$ (in particular, by
	Remark \ref{rc}, this holds if $-G(x^0,u^0)\ge\|H(x^0,u^0)\|e$) and that $y-x \geq \|v-u\|e$ implies
	\begin{equation*}
		y-x - G(y,v) + G(x,u) \geq \|v- u- H(y,v) + H(x,u)\|e.
	\end{equation*}
	Let 
	\begin{equation*}
		\Omega = \{(x,u) \in \mathbb{R}^p \times C : x- x^0 \geq \|u - u^0\|e \text{ ,  }G(x,u)- x^0 \geq \|H(x,u)-u^0\|e \}
	\end{equation*}
	and
	\[
	\begin{array}{rl}
		\Gamma = \{ (x,u) \in \mathbb{R}^p \times C: & \hspace{-2.5mm} x- x^0 \geq \|u - u^0\|e,\\ & \hspace{-2.5mm} G(x,u)-x^0 
		\geq \|u- u^0-P_C(u-H(x,u))\|e\}
	\end{array}
	\]
	Consider the following assertions
	\begin{enumerate}[(i)]
		\item\label{i2} $\Omega\ne\varnothing$,
		\item\label{ii2} $\Gamma\ne\varnothing$,
		\item\label{iii2} The sequence $\{(x^n,u^n)\}_{n\in\mathbb{N}}$ is convergent and its limit $(x^*,u^*)$ is a solution of 
			$VI(F,K)$. Moreover, $(x^*,u^*)$ is the smallest element of 
			$\Gamma$ with respect to the partial order defined by $L$. 
	\end{enumerate}
	Then, $\Omega\subset\Gamma$ and (\ref{i})$\implies$(\ref{ii})$\implies$(\ref{iii}).
\end{theorem}
\begin{proof}
Let $L$ be the extended Lorentz cone defined by \eqref{extlor}.
First observe that $K\cap (x^0+ L)\ne\varnothing$. By using the definition of the extended Lorentz cone, it is easy to verify that 
	$$\Omega=K\cap  ((x^0,u^0)+L) \cap F^{-1}(L)=\{z\in K\cap ((x^0,u^0)+L):F(z)\in L\}$$ and $$\Gamma=\{z\in K\cap ((x^0,u^0)+L):P_K(z-F(z))\le_L z\}.$$ Let $x,y\in\mathbb{R}^p$ and $u,v\in C$. Since 
	$y-x\ge\|v-u\|e$ implies $y-x-G(y,v)+G(x,u)\ge\|v-u-H(y,v)+H(x,u)\|e$, it follows that $I-F$ is $L$-isotone. Hence, by Proposition 
	\ref{psumm} (with $m=p+q$) and Lemma \ref{lmicp}, it follows that $\Omega\subset\Gamma$ and 
	(\ref{i2})$\implies$(\ref{ii2})$\implies$(\ref{iii2}). 
\end{proof}

\section{Unbounded box constrained variational inequalities}

Let $p,q$ be positive integers, $m=p+q$ and $K=\bt_{\ell=1}^m[a_\ell, b_\ell]$ be a box, where $a_\ell,b_\ell\in\R\cup\{-\infty,\infty\}$ and
$a_\ell<b_\ell$, for all $\ell\in\{1,\dots,m\}$.
The $i$-th entry of the projection function is (see Example 1.5.10 in \cite{FacchineiPang2003}):
\begin{equation}\label{mid}
	(P_K(x))_i=P_{[a_i, b_i]}(x_i)=\md(a_i , b_i, x_i)= 
	\begin{cases}
		a_i\textrm{ if }x_i \leq a_i, \\
		x_i\textrm{ if }a_i \leq x_i \leq b_i, \\
		\,b_i\textrm{ if }b_i \leq x_i.
	\end{cases}
\end{equation}
Let $B=\bt_{i=1}^{p}[a_i, b_i] \subseteq \mathbb{R}^p$ and $C= \bt_{j=1}^{q}[a_{p+j}, b_{p+j}] $. So we have
\begin{equation}\label{split}
	P_K(y,v)=(P_B(y),P_C(v))
\end{equation}
and the Picard iteration \eqref{picard} becomes 
\begin{equation}\label{miditer}
	x^{n+1}_i%
	=\md(a_i,b_i,(x^n-F(x^n))_i).
\end{equation}
Let $L$ be the extended Lorentz cone defined by \eqref{extlor}. The next proposition shows that the $L$-isotonicity of a box is equivalent to 
the box being a cylinder.
\begin{proposition}\label{pu}
	Let $L$ be the extended Lorentz cone defined by \eqref{extlor}. Then, the projection mapping $P_K$ 
	is $L$-isotone
	if and only if $K=\mathbb{R}^p \times C$ where $C=\bt_{j=1}^q[a_{p+j},b_{p+j}]$.
\end{proposition}
\begin{proof}
	The sufficiency follows easily from item 1 of Theorem 2 in \cite{NemethZhang} (repeated in the Preliminaries as Theorem \ref{tliso})
	For the sake of completeness we provide a proof here. Suppose that $B=\mathbb{R}^p$. If $(x,u)\le_L (y,v)$, that is, 
	$y-x \geq \|v-u\|e$, then by the nonexpansivity of $P_C$ we get \[P_B(y)-P_B(x)= y-x \geq \|v-u\|e \geq \|P_C(v)-P_C(u)\|e,\] which is 
	equivalent to \[P_K(x,u)=(P_B(x),P_C(u))\le_L (P_B(y),P_C(y))=P_K(y,v).\] Hence, $P_K$ is $L$-isotone.
	Although, the necessity could also be derived from item 3 of the same theorem, it is more clear to prove this 
	directly as follows. Suppose that $P_K$ is $L$-isotone. We need to prove that  
	$a_i = -\infty, b_i= \infty$, for any $i=1, \dots p$. Assume to the contrary, that there exist at least one $k \in \{1, \dots , p\}$ such
	that either $a_k$ or $b_k$ is a finite real number. 
	Assume that $b_k$ is a finite real number. The case $a_k$ is a finite real number can be treated similarly. Let $u$ and $v$ be two 
	different vectors in $C$. Then, $P_C(u)=u$ and $P_C(v)=v$. It is easy to choose 
	$x,y\in\R^p$ such that $y_i-x_i\ge\|v-u\|$ for all $i\in\{1,\dots,p\}$ and $b_k\le x_k\le y_k$. For example, we may choose 
	$x_i=\delta_{ik}b_k$ and $y_i=\delta_{ik}b_k+\|v-u\|$, for all $i=\{1,\dots,p\}$, where $\delta_{ik}$ is the Kronecker symbol. Then, 
	$(x,u)\le_L (y,v)$ and by 
	\eqref{mid} we have 
	$(P_K(y,v))_k=(P_K(x,u))_k$, or equivalently $(P_B(y))_k=(P_B(x))_k$. Hence, by \eqref{split} and the $L$-isotonicity of $P_K$ we get 
	\[0=(P_B(y)-P_B(x))_k\ge\|P_C(v)-P_C(u)\|=\|v-u\|>0,\]
	which is a contradiction.

	Hence, the results of Theorem \ref{tsumm} can be particularized to the set $K$ given by Proposition \ref{pu}, with the Picard iteration
	taking the form \eqref{miditer} and the function 'mid' given by \eqref{mid}. In the next section we will present an example for this 
	particularized result.
\end{proof}

\section{Numerical example}

Let $K= \mathbb{R}^2 \times C$ where $C= [-10,10] \times [-10,10]$. Let $L$ be the extended Lorentz cone defined by \eqref{extlor}.  Let $f_1(x,u)=1/12(x_1+\|u\|+12)$ and $f_2(x,u)=1/12(x_2+\|u\|-7.2)$.  Then it is easy to show that these two functions are $L$-monotone. Let $w_1=(1,1,1/6,1/3)$ and $w_2=(1,1,1/3,1/6)$ so $w_1$ and $w_2$ is in $L$. For any two vectors $(x,u)$ and $(y,v)$ in $K$, suppose $(x,u) \leq_L (y,v)$, we have $y_1-x_1 \geq \|v-u\| \geq \|u\|-\|v\|$ by triangle inequality. Hence,
\[f_1(y,v)-f_1(x,u)=\f1{12}(y_1-x_1-(\|u\|-\|v\|))\geq 0.\] Similarly, we can prove that if $(x,u) \leq_L (y,v)$, then $f_2(y,v)-f_2(x,u) \geq 0$. Provided that $K$ is 
convex, and $w_1,w_2 \in L$, if $(x,u) \leq_L (y,v)$ holds, then \[(f_1(y,v)-f_1(x,u))w_1+ (f_2(y,v)-f_2(x,u))w_2 \in L.\] Thus, 
$ f_1(x,u)w_1+ f_2(x,u)w_2 \leq_L f_1(y,v)w_1+ f_2(y,v)w_2$. Therefore, the mapping $f_1w_1+f_2w_2$ is $L$-isotone. Hence, choose the function
\begin{equation}\label{fdef}
(x-G,u-H)=f_1w_1+f_2w_2=\left( f_1+f_2, f_1+f_2, \frac{1}{6}f_1+\frac{1}{3}f_2,  \frac{1}{3}f_1+\frac{1}{6}f_2 \right),
\end{equation}
where $G$, $H$, $f_1$ and $f_2$ are considered in the point $(x,u)$.
It is necessary to check that all the conditions in Theorem \ref{tsumm} are satisfied. First, let $(x^0, u^0)= (43/30, 13/30, 2, 5)$ which is in 
$K= \mathbb{R}^2 \times C$. Consider $(x, u) = (15, 15, 6, 8)$. Thus,
\begin{equation*}
x- x^0 =\lf(\frac{407}{30}, \frac{437}{30}\rg) \geq (5, 5) =  \|(4, 3)\|(1, 1)=\|u-u^0\|e.
\end{equation*}
By the definition of the mapping $(G, H)$, we know that $G(x, u) - x_0= x- (f_1 + f_2)e -x^0$ and $H(x, u) - u^0= u - (1/6f_1 + 1/3 f_2, 1/3f_1 +
1/ 6 f_2) - u^0$. Then,
\begin{equation*}
G(x, u) - x_0 = (15, 15) - \frac{1}{12}(15 +15 +2 \times 10 +4.8)e - \lf(\frac{43}{30}, \frac{13}{30}\rg) =(9, 10)
\end{equation*}
 and 
\begin{equation*}
H(x, u) - u^0 = (6, 8) - \lf(\frac{121}{120}, \frac{91}{40}\rg) - (2, 5) = \lf(\frac{359}{120}, \frac{91}{40}\rg)
\end{equation*}
 Hence, we can easily see that:
\begin{equation*} 
G(x, u) - x_0 \geq \| (3, 3)\| e \geq \|H(x, u) - u^0\|e,
\end{equation*}
which shows that $(x, u) \in \Omega$ and $\Omega\ne\varnothing$. Therefore, the conclusions of Theorem \ref{tsumm} will apply for our example.

Now, we begin to solve the $VI$. Suppose that $(x,u)$ is its solution. Since $G(x,u)=0$, and
\begin{equation*}
x-G(x,u)=(f_1+f_2, f_1+f_2),
\end{equation*}
where $f_i=f_i(x,u),i=1,2$, we have $x_1=x_2=f_1+f_2$.  Moreover, since \[x_1=\f1{12}(x_1+x_2)+\f16\|u\|+0.4,\] we get
\begin{equation}\label{s1}
x_1=x_2= \frac{1}{5}\|u\|+\frac{12}{25}.
\end{equation}
Obviously, if $(x,u) \in \ker(F)$, that is, $G(x,u)=0$ and $H(x,u)=0$, then $(x,u)$ is a solution of the variational inequality. The equality 
$H(x,u)=0$ implies
\begin{equation}\label{s2}
\lf\{
	\begin{array}{rcl}
		u_1&=&\frac{1}{6}f_1+\frac{1}{3}f_2,\\\\

		u_2&=&\frac{1}{3}f_1+\frac{1}{6}f_2.
	\end{array}
	\rg.
\end{equation}

By using equations \eqref{s1}, equations \eqref{s2} become
\begin{equation}\label{ucon}
\lf\{
	\begin{array}{rcl}
		u_1&=&\frac{1}{20}\|u\|-\frac{1}{75},\\\\

		u_2&=&\frac{1}{20}\|u\|+\frac{19}{75}.
	\end{array}
	\rg.
\end{equation}
Hence, $u_2=u_1+ \frac{4}{15}$. By substituting back to \eqref{ucon}, we obtain 
\begin{equation*}
u_1\lf(u_1-\frac{28}{995}\rg) = 0.
\end{equation*}
However, the equation \eqref{ucon} can be transformed to
\begin{equation*}
19u_1+u_2=\|u\|
\end{equation*}
By squaring both sides, we get
\begin{equation*}
361u_1^2+u_2^2+38u_1u_2=u_1^2+u_2^2.
\end{equation*}
Hence,
\begin{equation}\label{ucon2}
2u_1(180u_1+19u_2)=0
\end{equation}
Thus, if $u_1= \frac{28}{995}$ , then since $u_2=u_1+ \frac{4}{15}$ equation \eqref{ucon2} will not hold. Therefore, the only solution in this
case is  
\begin{equation*}
(x,u)=\lf(\frac{8}{15}, \frac{8}{15}, 0,\frac{4}{15}\rg).
\end{equation*}
Next, suppose that the solution $(x,u)\notin\ker(F)$, that is, $H=(x,u)\ne0$. Then, by \eqref{vicartesian} and \eqref{s1}, we get
\begin{equation}\label{gvi}
(v_1-u_1)(300u_1-15 \|u\| + 4)+(v_2 -u_2)(300u_2-15 \|u\| - 76) \geq 0
\end{equation}

From the general theory of variational inequalities, it is known that $u$ should be on the boundary of $C$.  
Let $u_1=10$. By choosing $v_1<u_1$ and $v_2=u_2$ and using $\|u\|<20$, it can be seen that inequality \eqref{gvi} will not hold. 
Similarly, if $u_2=10$, by choosing $v_2<u_2$ and $v_1=u_1$, inequality \eqref{gvi} will not be satisfied. If $u_1=-10$, then by choosing 
$v_1>u_1$ and $v_2=u_2$, inequality \eqref{gvi} will not hold. Similarly $u_2 = -10$ will not lead to a solution. Hence, this case will not 
provide a new solution. Therefore, the only solution of the variational inequality is
\begin{equation*}
(x,u)=\lf(\frac{8}{15}, \frac{8}{15}, 0,\frac{4}{15}\rg).
\end{equation*}
obtained above.

The Picard iteration can be completed by using any spreadsheet software. Note that since the variational inequality is box constrained, the 
iteration in \eqref{miditer} will be calculated by using the \emph{median} function as it is shown in the following tables. More precisely, the 
initial point is given in the first row for $n=0$ and the later iterations are given in the following rows, where the columns below $u^n_1$ and
$u^n_2$ are obtained by the \emph{median} of the upper bound, lowerbound and $u^n-H(x^n,u^n)$. In the first table, we showed the iteration using 
the previous example where $(x^0, u^0)= (43/30, 13/30, 2, 5)$. In the other tables, we showed the iteration from different initial points in
different directions outside of $C$. It can be observed that the iterations in all tables converge to the same unique solution of the variational
inequality.

\begin{center}
\begin{tabular}{|c|c|c|c|c|c|c} 
\hline
{n } & $x^n_1$ & $x^n_2$  & $u^n_1$ & $ u^n_2$\\
\hline
0 & $\frac{43}{30}$ & $\frac{13}{30}$ & 2 & 5\\
\hline
1 & $\frac{1329}{865}$ & $\frac{1329}{865}$ & $\frac{244}{973}$ & $\frac{15}{29}$\\[0.5ex]
\hline
2 & $\frac{3}{4}$ & $\frac{494}{657}$ & $\frac{10}{183}$ & $\frac{9}{28}$\\[0.5ex]
\hline
3 & $\frac{40}{69}$ & $\frac{575}{992}$ & $\frac{5}{432}$ & $\frac{7}{26}$\\[0.5ex]
\hline
4 & $\frac{19}{35}$ & $\frac{284}{523}$ & 0 & $\frac{4}{15}$\\[0.5ex]
\hline
5 & $\frac{53}{99}$ & $\frac{53}{99}$ & 0 & $\frac{4}{15}$\\[0.5ex]
\hline
6 & $\frac{8}{15}$ & $\frac{253}{474}$ & 0 & $\frac{4}{15}$\\[0.5ex]
\hline
7 & $\frac{8}{15}$ & $\frac{407}{763}$ & 0 & $\frac{4}{15}$\\[0.5ex]
\hline
8 & $\frac{8}{15}$ & $\frac{8}{15}$ & 0 & $\frac{4}{15}$\\[0.5ex]
\hline
9 & $\frac{8}{15}$ & $\frac{8}{15}$ & 0 & $\frac{4}{15}$\\[0.5ex]
\hline
10 & $\frac{8}{15}$ & $\frac{8}{15}$ & 0 & $\frac{4}{15}$\\[0.5ex]
\hline
11 & $\frac{8}{15}$ & $\frac{8}{15}$ & 0 & $\frac{4}{15}$\\[0.5ex]
\hline
\end{tabular}
\end{center}

\begin{center}
	\hspace{21mm}
\begin{tabular}{|c|c|c|c|c|c|c} 
\hline
{n } & $x^n_1$ & $x^n_2$  & $u^n_1$ & $ u^n_2$\\
\hline
0 & -6 & -10 & 6& 11\\
\hline
1 & $\frac{121}{43}$ & $\frac{1251}{514}$ & $\frac{122}{511}$ & $\frac{47}{93}$\\[0.5ex]
\hline
2 & $\frac{63}{85}$ & $\frac{232}{313}$ & $\frac{29}{558}$ & $\frac{29}{91}$\\[0.5ex]
\hline
3 & $\frac{56}{97}$ & $\frac{474}{821}$ & $\frac{9}{818}$ & $\frac{5}{18}$\\[0.5ex]
\hline
4 & $\frac{51}{94}$ & $\frac{389}{717}$ & $\frac{2}{869}$ & $\frac{7}{26}$\\[0.5ex]
\hline
5 & $\frac{38}{71}$ & $\frac{372}{695}$ & 0 & $\frac{4}{15}$\\[0.5ex]
\hline
6 & $\frac{8}{15}$ & $\frac{356}{667}$ & 0 & $\frac{4}{15}$\\[0.5ex]
\hline
7 & $\frac{8}{15}$ & $\frac{423}{793}$ & 0 & $\frac{4}{15}$\\[0.5ex]
\hline
8 & $\frac{8}{15}$ & $\frac{8}{15}$ & 0 & $\frac{4}{15}$\\[0.5ex]
\hline
9 & $\frac{8}{15}$ & $\frac{8}{15}$ & 0 & $\frac{4}{15}$\\[0.5ex]
\hline
10 & $\frac{8}{15}$ & $\frac{8}{15}$ & 0 & $\frac{4}{15}$\\[0.5ex]
\hline
11 & $\frac{8}{15}$ & $\frac{8}{15}$ & 0 & $\frac{4}{15}$\\[0.5ex]
\hline
\end{tabular}
\newline
\end{center}

\begin{center}
\begin{tabular}{|c|c|c|c|c|c|c} 
\hline
{n } & $x^n_1$ & $x^n_2$  & $u^n_1$ & $ u^n_2$\\
\hline
0 & -5 & 4 & -12 & 7\\
\hline
1 & $\frac{115}{17}$ & $\frac{1187}{212}$ & $\frac{321}{952}$ & $\frac{32}{53}$\\[0.5ex]
\hline
2 & $\frac{63}{76}$ & $\frac{63}{76}$ & $\frac{32}{433}$ & $\frac{16}{47}$\\[0.5ex]
\hline
3 & $\frac{31}{52}$ & $\frac{127}{213}$ & $\frac{12}{763}$ & $\frac{24}{85}$\\[0.5ex]
\hline
4 & $\frac{47}{86}$ & $\frac{47}{86}$ & $\frac{2}{607}$ & $\frac{17}{63}$\\[0.5ex]
\hline
5 & $\frac{52}{97}$ & $\frac{52}{97}$ & 0 & $\frac{23}{86}$\\[0.5ex]
\hline
6 & $\frac{8}{15}$ & $\frac{63}{118}$ & 0 & $\frac{4}{15}$\\[0.5ex]
\hline
7 & $\frac{8}{15}$ & $\frac{295}{553}$ & 0 & $\frac{4}{15}$\\[0.5ex]
\hline
8 & $\frac{8}{15}$ & $\frac{8}{15}$ & 0 & $\frac{4}{15}$\\[0.5ex]
\hline
9 & $\frac{8}{15}$ & $\frac{8}{15}$ & 0 & $\frac{4}{15}$\\[0.5ex]
\hline
10 & $\frac{8}{15}$ & $\frac{8}{15}$ & 0 & $\frac{4}{15}$\\[0.5ex]
\hline
11 & $\frac{8}{15}$ & $\frac{8}{15}$ & 0 & $\frac{4}{15}$\\[0.5ex]
\hline
\end{tabular}
\end{center}

\begin{center}
\begin{tabular}{|c|c|c|c|c|c|c} 
\hline
{n } & $x^n_1$ & $x^n_2$  & $u^n_1$ & $ u^n_2$\\
\hline
0 & 8 & -19 & -9 & -15\\
\hline
1 & $\frac{424}{37}$ & $\frac{4109}{168}$ & $\frac{125}{866}$ & $\frac{113}{44}$\\[0.5ex]
\hline
2 & $\frac{141}{91}$ & $\frac{1296}{657}$ & $\frac{36}{157}$ & $\frac{1}{2}$\\[0.5ex]
\hline
3 & $\frac{11}{15}$ & $\frac{463}{713}$ & $\frac{39}{782}$ & $\frac{25}{79}$\\[0.5ex]
\hline
4 & $\frac{19}{33}$ & $\frac{96}{131}$ & $\frac{4}{379}$ & $\frac{23}{83}$\\[0.5ex]
\hline
5 & $\frac{45}{83}$ & $\frac{80}{139}$ &$ \frac{1}{453}$ & $\frac{25}{93}$\\[0.5ex]
\hline
6 & $\frac{38}{71}$ & $\frac{45}{83}$ & 0 & $\frac{4}{15}$\\[0.5ex]
\hline
7 & $\frac{8}{15}$ & $\frac{213}{398}$ & 0 & $\frac{4}{15}$\\[0.5ex]
\hline
8 & $\frac{8}{15}$ & $\frac{372}{697}$ & 0 & $\frac{4}{15}$\\[0.5ex]
\hline
9 & $\frac{8}{15}$ & $\frac{447}{838}$ & 0 & $\frac{4}{15}$\\[0.5ex]
\hline
10 & $\frac{8}{15}$ & $\frac{8}{15}$ & 0 & $\frac{4}{15}$\\[0.5ex]
\hline
11 & $\frac{8}{15}$ & $\frac{8}{15}$ & 0 & $\frac{4}{15}$\\[0.5ex]
\hline
\end{tabular}
\end{center}

%

\section{Conclusions}
In this paper we presented a Picard iteration for solving a fixed point problem equivalent to a variational inequality on a cylinder. The 
iteration is monotonically convergent to the solution of the variational inequality with respect to the partial order defined by an extended
Lorentz cone. The monotone convergence is based on the isotonicity of the projection onto a cylinder with respect to the partial order defined by
the extended Lorentz cone. We plan to consider the following general questions in the future:
\begin{enumerate}
	\item Given a cone $K$, determine all closed convex sets $C$ onto which the projection is isotone with respect to the partial order 
		defined by the cone.
	\item Given a closed convex set $C$, determine all cones $K$ such that the projection onto $C$ is isotone with to respect the partial 
		order defined by $K$.
	\item Determine the closed convex sets $C$ for which there exists a cone $K$, such that the projection onto $C$ is isotone with respect 
		to the partial order defined by $K$.
\end{enumerate}
Although the above questions are difficult to answer in general, any particular result about them can be important for solving complementarity 
problems and/or variational inequalities by using a monotone convergence. Moreover, any such result could be important in statistics as well, 
where the isotonicity of the projection may occur in various algorithms (see for example \cite{GuyaderJegouNemeth2012}).

\bibliographystyle{hieeetr}
\bibliography{bib}

\end{document}